%% file: DLW_JDE_Rev1-1.tex
\documentclass[12pt]{article}
\usepackage{amssymb,amsmath,amsfonts}
\usepackage{dsfont}
\usepackage{amsthm}
\usepackage{graphicx,color,subfig}
\usepackage{xspace}
\usepackage{mathrsfs}
\input mymacros.tex

\newtheorem{thm}{Theorem}

\newtheorem{lem}[thm]{Lemma}
\newtheorem{prop}[thm]{Proposition}

\theoremstyle{remark}
\newtheorem{rem}{Remark}
\newtheorem{defn}{Definition}

\newtheorem{acknowledgement}{Acknowledgement}
\newcommand{\EQ}[1]{\begin{equation} \begin{split} #1 \end{split} \end{equation}}
\newcommand{\na}{\nabla}

\newcommand{\De}{\Delta}
\newcommand{\Om}{\Omega}

\newcommand{\I}{\infty}

\renewcommand{\bar}{\overline}
\renewcommand{\hat}{\widehat}

\numberwithin{equation}{section}

\numberwithin{equation}{section}
\makeatletter
\def\keywords{
    \vspace{-0.5ex}
    \noindent
    \if@twocolumn
      \small{\bf  Keywords}\/---$\!$    \else
      \begin{center}\small\ {\bf Keywords}\end{center}\quotation\small
    \fi}
\def\endkeywords{\vspace{-0.2em}\par\if@twocolumn\else\endquotation\fi
    \normalsize\rm}
\makeatother

\begin{document}

\title{\vspace{-1cm}Energy decay  for linear dissipative wave equations\\  in exterior domains}

\author{Lassaad Aloui \footnote{ email: lassaad.aloui@fsg.rnu.tn}\\ {\small \it Department of Mathematics, Faculty of Sciences, }\\ {\small \it Northern Borders University, Arar, P.O. Box 1321, Kingdom of Saudi Arabia.}\\Slim Ibrahim\footnote{email: ibrahim@math.uvic.ca http://www.math.uvic.ca/~ibrahim/} 
\\ { \small \it Department of Mathematics and Statistics, University of Victoria}\\
 {\small \it PO Box 3060 STN CSC, Victoria, BC, V8P 5C3, Canada.} \\ Moez Khenissi\footnote{ email: moez.khenissi@fsg.rnu.tn}\\ {\small\it D\'epartement de Math\'ematiques, ESST de Hammam Sousse}\\
{\small \it Rue Lamine El Abbessi, 4011 Hammam Sousse,Universit\'e de Sousse, Tunisie.}
  }

\maketitle
\vspace{-0.7cm}
\begin{abstract}

In earlier works \cite{Al-Kh,Kh}, we have shown the uniform decay of the local energy of the damped wave equation in exterior domain when the damper is  spatially localized near captive rays. In order to have uniform decay of the total energy, the damper has also to act at space infinity. In this work, 
we establish uniform decay of both the local and global energies. The rates of decay turns out to be the same as those for the heat equation, which shows that an effective damper at space infinity strengthens the parabolic structure in the equation.
\end{abstract}
\vspace{-0.3cm}
\tableofcontents
\section{Introduction}   \label{S-1}

Let $\Om=\R^N\setminus\mathcal O$ be the exterior of a smooth compact obstacle $\mathcal O\subset \R^N$ with $N\geq 2$. In the sequel, fix a constant  $r_0>0$  such that $\mathcal{O}\subset B_{r_0}=\{x\in \R^N;|x|<r_0\}$, and consider the damped linear wave equation in $\Om$,

\EQ{
 u_{tt} + au_t -\De u  = 0,\label{DLW}}
where the solution $u(t,x):[0,\I)\times\Om\to\R$, and the damper $a\geq0$. Supplementing  \eqref{DLW} with Dirichlet type boundary conditions $u_{|\partial\Om\times(0,\infty)}=0$, one can easily see that the energy given by 
\EQ{
 E(u;t) = \frac{1}{2}\int_\Om (|\partial_t u(x,t)|^2 + |\na u(x,t)|^2)\; dx
 }
 is non-increasing
$$
\frac{d}{dt}E(u;t)= -\int_\Om a(x)|\partial_t u(x,t)|^2\;dx.
$$
Energy rate of decay is not in general uniform. Moreover, one can also define the local energy given for any $r>0$ by
\EQ{
 E_r(u;t) = \frac{1}{2}\int_{\Om\cap B_r } (|\partial_t u(x,t)|^2 + |\na u(x,t)|^2)\; dx.
 }
 In this paper we study the uniform rates of decay of the local and total energies of solutions of \eqref{DLW}. Assuming in addition that the damper acts in space infinity and satisfies the well known geometric control condition abbreviated as {\bf GCC} (see Definition \ref{GCC}), then we have the uniform decay of the local energy $E_r$. More precisely, we show the following two results.
\begin{thm}
\label{main thm 1}
Let $r_0$ be as above. Suppose that the damper $a(x)\in L^\infty (\Om)$ is smooth, satisfies the {\bf GCC} and  such that $a(x)=1$ for $|x|>r_0$. Let $r_1>0$, then there exists $c=c(r_0,r_1,\|a\|_{L^\infty})>0$ such that
any solution  $u$ of (\ref{DLW}) with initial data $(u_0,u_1)\in {\mathcal H}_{r_1}:=H^D_{r_1}\times L^2_{r_1}$ and Dirichlet boundary condition satisfies 

\EQ{E_{r_1}(u,t)\leq c 
(1+t)^{-N} E_{r_1}(u,0).
\label{loc.energy}
}   
\end{thm}

The spaces $H^D$ and $H^D_r$ are defined in the next section. The second result concerns the decay of the total energy. We have
\begin{thm}\label{main thm 2}
Under the hypotheses of Theorem \ref{main thm 1}, for all $(u_0,u_1)\in{\mathcal H}_{r_1}$, 
 there exists $c>0$ such
that 

\EQ{\|u(t)\|_{L^2}^2\leq c(1+t)^{-N/2}\label{tot.L^2}} and
\EQ{E(u,t)\leq c(1+t)^{-\theta},\label{tot.energy}}
where \EQ{\theta=min\{1+\frac{N}{2},\frac{3N}{4}\},} and $u$ a solution of (\ref{DLW}) with Dirichlet boundary condition.
\end{thm}
\begin{rem} Note that when $N\geq 4$, we have $\theta=1+\frac{N}{2}$ and therefore, the decay in \eqref{tot.energy} is optimal for initial data in $L^1\cap H^1$. See \eqref{Lp-Lq type decay} for $m=1$.
\end{rem}
The main novelty in the above two results resides in the assumptions on the damper $a(x)$ as in Theorem \ref{main thm 2}, we basically require almost the strict minimum: the damper to be effective near space infinity and in the region where trapped rays may exist. In \cite{Doulati}, only the decay of the total energy was investigated, and Daoulati has proven that for initial data $(u_0,u_1)$ in the energy space, one has \eqref{L2-bound,energy decay}. This result is proven using an energy method combined with microlocal analysis techniques. Observe that (for $N\geq 1$) our decay rate is sharper than \eqref{L2-bound,energy decay}. However, our hypothesis on the initial data is stronger than in \cite{Doulati} as we take them with compact support.\\

In \cite{Shib-Dan} Dan-Shibata studied the local energy decay for the damped wave equation \eqref{DLW}. More precisely, they have proven estimate \eqref{loc.energy} when the damper is taken constant and initial data with compact support. They improved an earlier result of Shibata \cite{Shib} by relaxing the smoothness assumption on the initial data. The method used in \cite{Shib-Dan,Shib} is based on the study of the resolvent. Our hypothesis on the support of the damper clearly relaxes the one in \cite{Shib-Dan,Shib}, and in this regard, Theorem 1 can be seen as a complete generalization of Dan-Shibata's result.\\ 
The local energy decay for the damped wave equation has been  also studied by Aloui-Khenissi \cite{Al-Kh}. More precisely,  they take a damper with compact support and satisfying the GCC. That led to an exponential decay of the local energy when the space dimension is odd and a polynomial decay like \eqref{loc.energy} when it is even. This is of course a decay stronger than ours for odd dimensions.  The difference in the decay rate is mainly due to the behavior of the cutoff resolvent for the low frequencies. The results in this paper also highlight the fact that the effectiveness of the damper at space infinity enhances the parabolic regime in \eqref{DLW}: for the total energy, we recover the optimal rate of decay of the energy of solutions of the heat equations. For the local energy, although the initial data are compactly supported, the diffusion phenomena slows down the decay rate compared to the case when the damper is not effective at space infinity.\\

The study of the uniform energy decay of solutions of the damped wave equation has a very long history, especially in the linear case. Before going any further, let us mention the most important works.
First, notice that if the damper is effective everywhere in $\Om$ i.e. when $a(x)=constant>0$ for all $x\in\Omega$, one can easily derive the following (quite standard) energy decay estimate and an $L^2$-bound for solutions of \eqref{DLW}
\EQ{
 E(u;t)\leq C I_0^2(1+t)^{-1},\qquad  \|u(t)\|_{L^2_x}^2
 \le C I_0^2,\label{L2-bound,energy decay}}
where $I_0^2=\|u_0\|_{H^1}^2+\|u_1\|_{L^2}^2$, and $H^1$ is the standard inhomogeneous Sobolev space. It is important to mention that estimate (\ref{L2-bound,energy decay}) does not in general imply a
uniform decay of $E(u,t)$. However, when $\Om$ is bounded, and thanks to Poincar\'e inequality, one can control the $L^2$ norm by Dirichlet norm and therefore \eqref{L2-bound,energy decay} implies an exponential energy decay.\\
In the case of the free space $\R^N$ and $a\equiv 1$, Matsumura \cite{Mat} obtained a precise
$L^p-L^q$ type decay estimates for solutions of (\ref{DLW}) 
\EQ{
 E(u;t)\leq C(1+t)^{-1-N(1/m-1/2)},\qquad  \|u(t)\|_{L^2_x}^2
 \le C(1+t)^{-N(1/m-1/2)},\label{Lp-Lq type decay}}
 where $(u_0,u_1)\in (H^1\bigcap L^m)\times(L^2\bigcap L^m)$ and $ m\in [1,2]$. It is worth noticing that the decay rate given in \eqref{Lp-Lq type decay} is the same one for the heat equation. Indeed,  it is well known (see  e. g. \cite{Ponce}
, Proposition 3.5.7) that 

 \EQ{\|\nabla^j e^{t\Delta}v_0\|_{L^2_x}^2
 \le Ct^{-j-N(1/m-1/2)}   \|v_0\|_{L^m}^2,\qquad j=0,1.
 \label{nabla Lp-Lq type decay}}
The case $m=2$ in \eqref{Lp-Lq type decay} and \eqref{nabla Lp-Lq type decay} shows that we have only the boundedness of
$\|u(t)\|_{L^2}$ and $\|e^{t\Delta}v_0\|_{L^2}$. In \cite{Ikeh02}, Ikehata has discussed the diffusion phenomenon for the damped wave equation in exterior domains. More precisely he proved
\EQ{\int_0^{+\infty}\|u(t,.)-v(t,.)\|_{L^2}^2dt\leq
\frac{4}{3}E(u,0),\qquad (u_0,u_1)\in H^1\times L^2}
and for all $(u_0,u_1)\in H^2\times H^1$
\EQ{\|u(t,.)-e^{t\Delta} (u_0+u_1)   \|_{L^2}\leq
C\frac{(\|u_0\|_{H^2}+\|u_1\|_{H^1})}{t^{1/2}{\rm log}(1+t)},
\label{asymp}}
Thus we can say that the damped wave $u(t,x)$ is asymptotically
equal to some solution $v(t,x)$ of the heat equation as $t\rightarrow +\infty$. The diffusion phenomena has also been discussed by Nishihara \cite{Nishihara97} and  Han and Milani \cite{Milani-Han} for
the Cauchy problem cases in $\mathbb{R}$ and $\mathbb{R}^N$, respectively. A more precise asymptotic expansion of the solution of the damped equation was given by \cite{Nishihara03} in
$\mathbb{R}^3$, and  Hosono and Ogawa \cite{Hos-Ogawa} in
$\mathbb{R}^2$:

\EQ{\|u(t)-e^{t\Delta} (u_0+u_1)-e^{-t/2}\{w(t)+(\frac{1}{2}+\frac{t}{8})\widetilde{w}(t)\}\|_{L^p}\leq
Ct^{-3/2(1/q-1/p)-1}(\|u_0\|_{L^q}+\|u_1\|_{L^q})} where $w(t,x)$
and $\widetilde{w}(t,x)$ are the solutions of the linear equation
with the initial data  $w(0,.)=u_0$, $\partial _tw(0,.)=u_1$ and
$\widetilde{w}(0,.)=0$, $\partial_t\widetilde{w}(0,.)=u_0.$\\
For the case of the exterior of a star shape obstacle and when $a(x)$ is uniformly positive near infinity, Mochizuki and Nakazawa \cite{Moch-Nak} have derived
$L^2$-bound and energy decay like (\ref{L2-bound,energy decay}). Nakao \cite{Nakao}  established (\ref{L2-bound,energy decay}) for trapping domains under the assumption that the dissipation is effective near space infinity and on a part of the boundary. Recently, in the same context, Daoulati \cite{Doulati} showed the same result by assuming that each trapped ray meets the damping region which is also effective at space infinity.\\ 
Another direction of investigation was developed considering weighted initial data. In \cite{Ikeh-Mats02}, Ikehata and Matsuyama have derived, when $a(x)\equiv1$, a decay rate like
\EQ{
 E(u;t)\leq C(1+t)^{-2},\qquad  \|u(t)\|_{L^2_x}^2\le C(1+t)^{-1},\label{fast decay}}
which is faster than \eqref{L2-bound,energy decay}. In \cite{Ryo03}, Ikehata obtained the estimate \eqref{fast decay} for solutions of the system \eqref{DLW} with weighted initial data and assuming that $a (x)\geq a_0 > 0$ at infinity and $O = \mathbb{R}^N \setminus\Omega$ is star shaped with respect to the origin. This result has been improved by Daoulatli \cite{Doulati} under the geometric control condition on the damping region, without any assumption on the shape of the boundary. Clearly, from the results of \cite{Doulati,Ryo03}, our estimate in Theorem \ref{main thm 2}  is not  optimal in the case $N=2$.\\ 
Furthermore, Zuazua \cite{Zua90} treated the nonlinear Kleine-Gordon
equations with dissipative term. He derived the exponential decay of energy through the weighted energy method. This result was later on generalized by Aloui, Ibrahim and Nakanishi \cite{Al-Ib-Na} to more general nonlinearities.

This paper is organized as follows. In section two, we introduce the notation we use along the paper, and recall basic definitions and a few lemmas necessary to prove our results. To derive the uniform decay of the local energy, we use a crucial result due to Shibata. In order to be able to use this result, we need to estimate the resolvent  in the case of low frequencies. This is the content of subsection 3.1.  In subsections 3.2 and 3.3 we estimate the resolvent estimates in the case of intermediate and high frequencies, respectively. The uniform local and total energy decays are proved in section 4 and section 5, respectively. 


\section{Notation and useful tools}   \label{S-2}

For $r>0$, recall that $B_r$ is the ball of radius $r$ and centered at the origin, and let $\Omega_r$ denote the set $\Om\cap B_r$.\\
For any functional space $X$ defined on $\Om$, denote by $X_r$ the set of functions in $X$ with support in  $\Om_r$. $X_{\text{comp}}$ denotes the space of function in $X$ with compact support and $X_{\text{loc}}$ the set of functions $u$ in $X$ such that $\varphi u\in X$, for all smooth $\varphi$ with support in $\Om$. For $k=1,2\cdot\cdot$, the space $H^k(\Omega )$ denotes the standard Sobolev space of order $k$. Let 
$ H^D(\Omega )$ denote the completion of $ C_0^\infty(\Omega)$ with respect to  the norm $\|\nabla\cdot\|_{L^2}$, and define
\begin{eqnarray}
\label{space H}
{\mathcal H}:=(H^D\times L^2)(\Omega),\quad\mbox{with the norm}\quad \|(u_1,u_2)\|_{\mathcal H}:=\|\nabla u_1\|_{L^2}+\|u_2\|_{L^2}.
\end{eqnarray}
\begin{defn}
\label{GCC}\cite{le}
We say that the damper $a$ satisfies the geometric control condition {\bf GCC} if there exists $T_0 > 0$ such that from every point in $\Omega$ the generalized geodesic meets the set
$\omega =\left\{ x\in \Omega ,a(x)>0\right\} $ in a time $t < T_0$.

\end{defn}
This definition was extended by Aloui-Khenissi \cite{Al-Kh} to the case of exterior domains. 

\begin{defn}\cite{Al-Kh}
\label{EGC}
We say that the damper $a$ satisfies the exterior geometric control condition {\bf EGC} if every generalized geodesic
has to leave any fixed bounded region or meet the damping region $\omega =\left\{ x\in \Omega ,a(x)>0\right\} $ in uniform finite time.

\end{defn}

For two Banach spaces $X$ and $Y$  let $\mathcal{L}(X,Y)$ denote the set of all
bounded linear operators from $X$ into $Y$ and $\mathcal{L}(X):=\mathcal{L}%
(X,X)$. For any domain $Q$ in $\mathbb{C}$, Anal $(Q,X)$ denotes the set of
all holomorphic functions defined on $Q$ with their values in X.

\begin{defn}[see \cite{Shib}]    

For a Banach space $\mathcal B$, $k=1,2\cdot\cdot$ and $0<\sigma\leq1$, let $\alpha=k+\sigma$ and $\mathcal C^\alpha=\mathcal C^\alpha(\R;\mathcal B)$ endowed with the norm $\|\cdot\|_{\alpha,\mathcal B}$ of the usual Besov space $B^\alpha_{1,\infty}$ given by: If $0<\sigma<1$

$$
\|f\|_{\alpha,\mathcal B}=\sum_{j=0}^k\int_\R\|\big(\frac d{d\tau}\big)^ju(\tau)\|_{\mathcal B}\;d\tau+\sup_{h\neq0}|h|^{-\sigma}   \int_\R\|\big(\frac d{d\tau}\big)^ku(\tau+h)-\big(\frac d{d\tau}\big)^ku(\tau)\|_{\mathcal B}\;d\tau.
$$
If $\sigma=1$,

\begin{eqnarray*}
\|f\|_{\alpha,\mathcal B}&=&\sum_{j=0}^k\int_\R\|\big(\frac d{d\tau}\big)^ju(\tau)\|_{\mathcal B}\;d\tau\\&+&\sup_{h\neq0}|h|^{-1}   \int_\R\|\big(\frac d{d\tau}\big)^ku(\tau+2h)-2\big(\frac d{d\tau}\big)^ku(\tau+h)+\big(\frac d{d\tau}\big)^ku(\tau)\|_{\mathcal B}\;d\tau.
\end{eqnarray*}

\end{defn}

In our analysis, a crucial role will be played by the resolvent of the Dirichlet Laplace operator in exterior domain. More precisely, let $\hat u$ solve the following exterior Dirichlet problem%

\begin{eqnarray}\label{resolvent Delta}
\left\{
\begin{tabular}{l}
$(\tau -\Delta )\hat u=f\text{ \ in }\Omega $  \\
$\hat u_{|\partial \Omega }=0,$
\end{tabular}
\right.
\end{eqnarray}
and denote by $A(\tau ):=(\tau -\Delta )^{-1}$ be its resolvent defined for all $\tau$ in the set
$$
\tau \in S_{\rho,\varepsilon }=\{\tau \in\mathbb C\backslash \{0\};\;|\tau |<\rho,\;|\arg \tau |<\pi -\varepsilon \},\;0<\rho<1\quad\mbox{and}\quad0<\varepsilon <\pi /2.
$$
For small values of $\tau$, the behavior of the resolvent $A(\tau)$ is given by the following lemma due to Vainberg \cite[Theorem 2 and Theorem 3]{vainb} (see also \cite{tsut}).

\begin{lem}\label{vainb lem}
Let $N\geq 3.$ There exists $\rho$ and $\varepsilon $ such that $%
A\in $Anal$(S_{\rho,\varepsilon };\mathcal L(L^2_{\text{comp}},H^{2}(\Omega ))$ and for $0<r'<r$, we have:

\begin{enumerate}
\item If $N$ is odd, then we have%
\begin{equation}
A(\tau)=\sum_{j=0}^\infty B_{j}\tau^j+\tau ^{\frac{N-2}{2}}\sum_{j=0}^\infty B_{j}'\tau^j \text{ }
\label{pai1}
\end{equation}%
in the region $S_{\rho,\varepsilon } $, where the operators $B_{j}$ and $B'_{j}$ ( $j=0,1,\cdots )$ are bounded linear operators from $L_{r}^{2}(\Omega)$ to $H^{2}(\Omega _{r^{\prime }})$ and the expansion (\ref{pai1}) converges uniformly and absolutely in the operator norm. 

\item If $N$ is even, then we have 
\begin{equation}
A(\tau )=\sum_{j=0}^\infty \tilde{B}_{j}\tau^j+\sum_{m=1}^\infty (\tau ^{\frac{N-2}{2}}\log \sqrt{\tau })^m\tilde{B}_{m}'(\tau)  \label{imp1}
\end{equation}%
in the region $ S_{\rho,\varepsilon } $, 
where the operators $\tilde{B}_{j}$  ( $j=0,1,\cdots )$ are bounded linear operators from $L_{r}^{2}(\Omega)$ to $H^{2}(\Omega _{r^{\prime }})$ and for $m=1,2,\cdots$, $\tilde{B}_{m}'(\tau)$ is in Anal($\{|\tau|<\varepsilon\};\mathcal L(L_{r}^{2}(\Omega ),H^{2}(\Omega _{r^{\prime
}}))$ and the expansion (\ref{imp1}) converges uniformly and absolutely in the operator norm.
\end{enumerate}
\end{lem}
In order to establish a relation between the resolvent $A$ and the resolvent associated to \eqref{DLW}, let us first consider $u$ solving \eqref{DLW} with the initial data $u(0)=0$ and $\partial_tu(0)=f$. Then define
\begin{equation}
R_a(\lambda )f=\int_{0}^{+\infty }e^{-\lambda t}u(t)dt,\text{ \ }\mbox{ Re }%
\lambda >0,  \label{semi}
\end{equation}
It is clear that the relation \eqref{semi} defines a family of bounded linear operators from $L^2(\Omega )$ to $L^{2}(\Omega )$, analytic on the set $\mathbb C^+:=\ \left\{ \lambda \in \mathbb{C;}\text{ }\mbox{
Re}\;\lambda >0\right\} $. Moreover, when the damper $a$ is effective everywhere i.e. $a\equiv 1$, then we have the following obvious relation

\begin{eqnarray}
\label{relation1}
A(\lambda(\lambda+1))=R_1(\lambda).\quad 
\end{eqnarray}
In the general case when $a(x)=1$ only for all $|x|\geq r_0$, consider $\chi \in C_{0}^{\infty }(\mathbb{R}^{N}),$ and define 
the truncated resolvent
\begin{equation*}
R_{a,{\chi }}(\lambda )=\chi R_a(\lambda )\chi
\end{equation*}%
seen as an  analytic  function on $\mathbb C^+$ with values in the space of bounded operators from $L^2(\Omega )$ to $L^{2}(\Omega )$.
In order to reformulate problem \eqref{DLW} in the semigroup setting, we set $v = u_t$ and $U=\left(
\begin{tabular}{l}
$u$ \\ $v$ 
\end{tabular}
\right)$. Then, in matrix form \eqref{DLW} reads 
\begin{equation}\label{semig}
\left\{
\begin{tabular}{ll}
$(\partial _{t}-B_a)U(t)=0$ & on $\mathbb{R}_{+}\times \Omega $ \\
$U_{|_{\partial \Omega }}=0$ & on $\mathbb{R}_{+}$ \\
$U(0)=U_{0}\in \mathcal H,$ &
\end{tabular}
\right.
\end{equation}
where we have set

\begin{equation*}
B_a=\left(
\begin{tabular}{ll}
$0$ & $I$ \\
$\Delta $ & $-aI$%
\end{tabular}
\right).
\end{equation*}
The relation between the resolvents $(\tau-B_a)^{-1}$ and $R_a$ is given by (see \cite{Kh} )

\begin{equation}
\label{relation2}
\mathcal{R}(\lambda ):=(\lambda I- B_a)^{-1}=\left( 
\begin{array}{cc}
 R_a(\lambda )(\lambda +a) & R_a(\lambda ) \\ 
  R_a(\lambda)\lambda(\lambda +a)-1 & \lambda R_a(\lambda )%
\end{array}%
\right) .
\end{equation}
One of the crucial steps in our proof is to show that the resolvent $R_a(\lambda)$ is in the class $\mathcal C^\frac{ N}{2}$ with uniform estimates with respect to the real part $\alpha$ of $\lambda$ when $\lambda=\alpha+is$ is in a compact set. To have that, we use the following important result due to Shibata \cite{Shib} (see Theorem 3.2) which 

\begin{lem}
\label{Shihb lem}
Let $p$ be an integer and $I=(-2,2)$. Consider a Banach space   $(\mathcal B,|\cdot|)$, and assume that $f\in\mathcal C^\infty(\R\setminus0;\mathcal B)\cap\mathcal C^{p-1}_0(I;\mathcal B)$ with 
\begin{itemize}
\item Assume $k=p+\sigma$ with $0<\sigma<1$ and $f$ satisfies: for $\tau\in I\setminus0$
$$
|\big(\frac d{d\tau}\big)^jf(\tau)|_{\mathcal B}\leq C(f),\quad j=0,1,\cdot\cdot p-1,
$$

$$
|\big(\frac d{d\tau}\big)^pf(\tau)|_{\mathcal B}\leq C(f)|\tau|^{\sigma-1},\quad |\big(\frac d{d\tau}\big)^{p+1}f(\tau)|_{\mathcal B}\leq C(f)|\tau|^{\sigma-2}.
$$
Then, $f\in\mathcal C^k(\mathbb R;\mathcal B)$ and
\begin{eqnarray}
\label{unif est1}
\|f\|_{k,\mathcal B}\leq C(\sigma,p)C(f).
\end{eqnarray}
\item  Assume $k=p+1$ and for $\tau\in I\setminus0$ $f$ satisfies $\big(\frac d{d\tau}\big)^pf(\tau)=f_0\log|\tau|+f_1$ and
$$
|\big(\frac d{d\tau}\big)^ju(\tau)|_{\mathcal B}\leq C(f),\quad j=0,1,\cdot\cdot p-1,
$$

$$
|f_0|_{\mathcal B}+|f_1|_{\mathcal B} \leq C(f),\quad  |\big(\frac d{d\tau}\big)^{p+j}f(\tau)|_{\mathcal B}\leq C(f)|\tau|^{-j},\quad j=1,2.
$$
Then, $f\in\mathcal C^k(\mathbb R;\mathcal B)$ and
\begin{eqnarray}
\label{unif est11}
\|f\|_{k,\mathcal B}\leq C(\sigma,p)C(f),\quad |f(\tau)-f(0)|_{\mathcal B}\leq C(p,\sigma)c(f)|\tau|^\frac12.
\end{eqnarray}
\end{itemize}
Here $C(\cdot)$ denotes a constant that may depend on its argument.
\end{lem}


\section{The resolvent estimates}
In this section, we estimate the resolvent. The analysis is done in three steps.
\subsection{Low frequencies}
First we derive the behavior of the resolvent near $\lambda\sim0$. It turns out that in this case, the asymptotic is the same as the one corresponding to the constant damper  $a(x)\equiv1$. More precisely, we have the following lemma.

\begin{lem}
\label{resolvent reg}
Let $N\geq 2.$ Then there exist  $\rho'>0$ and $\varepsilon'>0$ such that the resolvent 
$$
R_a(\lambda )\in {\rm Anal}(S_{\rho',\varepsilon '};\mathcal L(L^2_{comp},H^{2}(\Omega ))
$$ 
and

\begin{enumerate}
\item if $N$ is odd, then%
\begin{equation}
R_a(\lambda )=G_{1}(\lambda )+\lambda ^{\frac{N-2}{2}}G_{2}(\lambda )\text{ }
\label{pai2}
\end{equation}%
in the region $S_{\rho',\varepsilon '}$, where $G_{1}(\cdot )$ and $G_{2}(\cdot )$ are in 
$$
{\rm Anal}(\{\lambda \in \mathbb{C}:\; |\lambda|<\varepsilon' \},\mathcal L(L_{r}^{2}(\Omega);H^{2}(\Omega _{r^{\prime }}))).
$$ 

\item if $N$ is even, then we have 
\begin{equation}
R_a(\lambda )=G_{3}(\lambda )+\sum_{m=1}^\infty (\lambda^{\frac{N-2}{2}}\log \sqrt{\lambda })^m\tilde{G}_{m}'(\lambda)  \label{imp2}
\end{equation}%
in the region $S_{\rho',\varepsilon '},$ 
where $G_{3}(\cdot )$ and $\tilde{G}_{m}'(\cdot )$, ($m=1,2,\cdots$), are in 
$$
{\rm Anal}(\{\lambda \in \mathbb{C}:\; |\lambda|<\varepsilon' \},\mathcal L(L_{r}^{2}(\Omega);H^{2}(\Omega _{r^{\prime }}))).
$$
\end{enumerate}
\end{lem}

\begin{proof}
The proof uses Lemma \ref{vainb lem} and the idea is to establish the relation between the resolvents $R_1$ and $R_a$. Observe that 
 $u:={R_1}(\lambda )f$ can be written as the solution of the system
\begin{equation*}
\left\{ 
\begin{array}{l}
(-\Delta +\lambda ^{2}+\lambda a(x))u=f+\lambda (a(x)-1)u \\ 
u_{|\partial \Omega }=0.%
\end{array}%
\right. 
\end{equation*}%
Hence
\begin{equation*}
\begin{array}{ll}
{R_1}(\lambda )f=u 
& =R_a(\lambda )(f+\lambda (a(x)-1){R_1}(\lambda )f) \\ 
& =R_a(\lambda )[I+\lambda (a(x)-1){R_1}(\lambda )]f.%
\end{array}%
\end{equation*}%
Since $a(x)$ $\equiv 1$ on the complement of the ball $B_{r_{0}}$, we deduce that for all $\chi \in C_{0}^{\infty }(\Omega )$ such that $\chi
\equiv 1$\ on $B_{r_0}$ and for all $f\in L^{2} $ supported on $B_{r}$ with $r>r_0$, we have

\begin{eqnarray}
\label{relation3}
{R}_{a,\chi }(\lambda )=R_{1,\chi }(\lambda )[I+\lambda (a(x)-1)%
{R}_{1,\chi }(\lambda )]^{-1}.
\end{eqnarray}%
The proof is now immediate using Lemma  \ref{vainb lem} and relations \eqref{relation1} and \eqref{relation2}. 
Indeed, we take $r^{\prime }<r<1$ and $\varepsilon^{\prime }<\pi /2$ small enough such that  $\lambda (\lambda +1)\in S_{r,\varepsilon }$ when
$\lambda \in S_{r^{\prime },\varepsilon ^{\prime }}$. 

\end{proof}

For any $\delta>0$, put $Q_\delta=\{\lambda\in\mathbb{C}; 0<Re\lambda<\delta,|Im\lambda|<\delta\}.$ Then we have the following result.

\begin{lem}\label{low.freq}

There exists $\delta>0$, such that
\begin{enumerate}
    \item $\mathcal{R}(\lambda)\in Anal(Q_\delta;\mathcal{L}(\mathcal{H}_{com},\mathcal{H}_{loc}))$
    \item For $\rho(s)\in C^\infty_0(\mathbb{R})$
    such that $\rho(s)=1$ if $|s|<\delta/2$ and $\rho=0$ if $|s|>\delta$, there
    exists $M_1>0$ depending on $r,\rho$ and $\delta$ such that for any $f\in \mathcal{H}_r$ and
$0<\beta<\delta$, we have
    \EQ{\| \rho(.)\chi\mathcal{R}(\beta+i.)f\|_{N/2} \leq
M_1\|f\|_{\mathcal H},}
where we set $\|\cdot\|_\alpha:=\|\cdot\|_{\alpha,\mathcal H}$.
\end{enumerate}

\end{lem}

\begin{proof} 
Since $G_1$, $G_2$, $G_3$ and $\tilde{G}_m$ are analytic, then in what follows, we can omit the computation of the norm in $\mathcal H$. Thanks to \eqref{relation2} and using the asymptotic \eqref{pai2}, \eqref{imp2}, it is sufficient to satisfy the conditions of Lemma \ref{Shihb lem} for the truncated resolvent 
$$
\rho(s)(\beta+is)^\frac{N-2}2, \quad\mbox{and}\quad \rho(s)\bigg((\beta+is)^\frac{N-2}2\log(\beta+is)\bigg)^m.
$$
Indeed, if $\frac{N}{2}=p+\frac12$ and for $u(s):= \rho(s)(\beta+is)^{\frac{N}{2}-1}$ we have

$$
\sup_{0\leq\beta\leq1}|\big(\frac d{ds}\big)^ju(s)|_{\mathcal H}\leq C(N,\delta),\quad j=0,1,\cdot\cdot p-1,
$$
and 
$$
|\big(\frac d{ds}\big)^{p}u(s)|_{\mathcal H}\leq C(N)\sup_{|\beta|\leq\delta}\big(|\beta+is|^{p-\frac12}+|s|^{-\frac{1}2}\big) \leq C(N,\delta) |s|^{-\frac{1}2},
$$
$$
|\big(\frac d{ds}\big)^{p+1}u(s)|_{\mathcal H}\leq C(N) \sup_{|\beta|\leq\delta}\big(|\beta+is|^{p-\frac12}+|s|^{-\frac{3}2}\big) \leq C(N,\delta) |s|^{-\frac32}.
$$
If $N$ is even, we write $\frac N2=p+1$ and we only estimate the function $u_1(\tau):=\tau^p\log|\tau|$ as all the other terms behave better. Clearly we have $(\frac{d}{d\tau})^pu_1(\tau)=C_1\log|\tau|+C_2$ and for $j=0,1,\cdot\cdot p-1$ 

$$
(\frac{d}{ds})^ju_1(\beta+is)=P_1(\beta+is)(\beta+is)\log|\beta+is|+P_2(\beta+is),
$$
where $P_1$ and $P_2$ are two polynomials. This implies that for all $j=0,1,\cdot\cdot p-1$

$$
\sup_{0\leq\beta\leq1}|\big(\frac d{ds}\big)^ju_1(\beta+is)|_{\mathcal H}\leq C(N,\delta),\quad j=0,1,\cdot\cdot p-1,
$$
and for $j=1,2$
$$
|\big(\frac d{ds}\big)^{p+j}u_1(\beta+is)|_{\mathcal H}\leq C(N) \sup_{|\beta|\leq\delta}\big(|\beta+is|\log|\beta+is|+|s|^{-j}\big) \leq C(N,\delta) |s|^{-j}
$$
which completes the proof.
\end{proof}


\subsection{Intermediate frequencies}
Recall that $a=1$ in $|x|>r_0$. Consider the reduced equation
\EQ{-\Delta
 w-s^2w+is aw=F\label{eq.reduite}.
}
We have the following {\it \`a priori} estimate for \eqref{eq.reduite}.
 \begin{prop}
 \label{moy.freq} 
 Let $\eta,M>0$.  Then, there exists  $C>0$ such that for any $w$ solution of \eqref{eq.reduite} and a real  number $s$ satisfying $\eta<s<M$, we have 
 \EQ{\|\nabla w\|_{L^2}^2+\|w\|^2_{L^2}\leq
C\|F\|_{L^2}^2.}
In particular, \eqref{eq.reduite} has a unique solution.
\end{prop}

\begin{proof}
Let $r>r_0+2$ be large and $\chi_1\in
C^\infty_0(\Omega)$, $\chi_1=0$ for $|x|>r+1$ and $\chi_1=1$ for
$|x|<r$. The function $v=\chi_1 w$ satisfies 
\EQ{\label{trunc u}
-\Delta v-s^2v+isa_1v=\chi_1 F-\nabla w\nabla \chi_1-\Delta \chi_1 w.
} 
First, notice that the damper $a_1=\chi_1 a$ satisfies the {\bf GGC}. Then from \cite{Al-Kh} we have
\EQ{\|\nabla v\|_{L^2}^2+s^2\| v\|^2_{L^2}\leq C\|\chi_1 F-\nabla
w\nabla \chi_1-\Delta \chi_1 w\|_{L^2}^2} which in particular implies \EQ{\|\nabla
w\|_{L^2_r}^2+s^2\| w\|^2_{L^2_r}\leq C\|\chi_1 F\|_{L^2}+\|\nabla
w\|_{L^2(r<|x|<r+1)}+\|
w\|_{L^2(r<|x|<r+1)}\label{loc.resol.est1}.}
Next, our aim is to estimate $\|\nabla w\|_{L^2(r<|x|<r+1)}$ and $\|
w\|_{L^2(r<|x|<r+1)}$. Multiplying equation \eqref{eq.reduite} by $\bar{w}$ and integrating we obtain the identity
\EQ{\|\nabla
w\|_{L^2}^2-s^2\|w\|_{L^2}+is\int a|w|^2dx=\langle F, w\rangle
\label{mult.u}}
for which Young's inequality applied to its imaginary part gives 
\EQ{s\int a|w|^2dx\leq
\eta \| w\|_{L^2}^2+C_\eta\|F\|_{L^2}^2 \label{au},}
for some $C_\eta>0$ depending only on $\eta$ and which may change from line to line. Now, let $\varphi\in C^\infty(\Omega)$ be a real valued function such that $\varphi=0$ for $|x|<r-2$ and $\varphi=1$ for $|x|>r-1$. Again, multiplying equation \eqref{eq.reduite} by
 $\varphi \Delta \bar{w}$ and integrating,  we get
 \EQ{
 -\int \varphi |\Delta w|^2dx+s^2\int \varphi|\nabla w|^2dx-is\int \varphi|\nabla w|^2 dx=\langle F,\varphi \Delta w\rangle-
 (s^2-is)\langle w,\nabla \varphi \nabla w\rangle \label{mult.delta.u}.}
Similarly, from the imaginary part we obtain 
\EQ{\int \varphi |\nabla
w|^2dx\leq \eta \|\varphi\Delta
w\|_{L^2}^2+C_\eta\|F\|_{L^2}^2+\eta \| \nabla w\|_{L^2}^2+C_\eta\|
w\|_{L^2(r-2<|x|<r-1)}^2,\label{ext.grad.u1}
}
and (\ref{au}) gives 

\EQ{\int \varphi |\nabla
w|^2dx\leq \eta \int \varphi |\Delta
w|^2dx+\|F\|_{L^2}^2+\eta \| \nabla w\|_{L^2}^2+\eta
\| w\|_{L^2}^2. \label{ext.grad.u2}}
Now, taking the real part of (\ref{mult.delta.u}), we obtain 

\EQ{\int
\varphi |\Delta w|^2dx\lesssim \| \nabla w\|_{L^2}^2+\|F\|_{L^2}^2+\|
w\|_{L^2}^2. \label{ext.delta.u}}
Combining (\ref{ext.grad.u2}) and (\ref{ext.delta.u}) yields

\EQ{\int \varphi |\nabla w|^2dx\lesssim
 \|F\|_{L^2}^2+\eta \| \nabla
w\|_{L^2}^2+\eta \| w\|_{L^2}^2. \label{ext.grad.u3}}
Thus, from (\ref{au}) and (\ref{ext.grad.u3}) we deduce that
\EQ{\int_{|x|>r}  (|\nabla w|^2+|w|^2)dx\leq
 \|F\|_{L^2}^2+\eta (\|\nabla w\|_{L^2_r}^2+\| w\|^2_{L^2_r}). \label{ext.res.est}}
Finally, by (\ref{loc.resol.est1})  and (\ref{ext.res.est}), we get

\EQ{\|\nabla w\|_{L^2_r}^2+\| w\|^2_{L^2_r}\lesssim \|F\|_{L^2}^2}
which implies \EQ{\|\nabla w\|_{L^2}^2+\|
w\|^2_{L^2}\lesssim \|F\|_{L^2}^2,}
as desired.
\end{proof}
\begin{rem}
Note that our resolvent estimate  is global that is without an assumption on
the support of the initial data.
\end{rem}
\subsection{High frequencies }
The following result studies the resolvent in the case of high frequencies.  It is easy to prove the following result
\begin{prop}\label{high.freq}
There exist two constants $\beta,c>0$ such that for all $s\in\R;|s|>\beta$, the solution $w$ of (\ref{eq.reduite}) satisfies 

\EQ{\| w\|_{L^2}\leq
\frac{c}{s}\|F\|_{L^2}.}
\end{prop}
\begin{proof} We argue by contradiction assuming that there exist $s_n>n$, $w_n\in L^2$
solving of (\ref{eq.reduite}), $F_n\in L^2$ such that 

\EQ{\|
w_n\|_{L^2}=1,\quad\mbox{and}\quad\frac{1}{s_n}\|F_n\|_{L^2}\rightarrow 0.\label{lim}} 
Then we can associate to the sequence $(w_n)$ a micro-local defect measure $\nu$ (we refer for example to \cite{le,Al-Kh} for the definition).
Set
$h_n=\frac{1}{s_n}$, we have \EQ{-h_n^2\Delta
w_n-w_n+ih_naw_n=h_n^2F\label{eq.reduite.semicl}} 
Multiplying equation \eqref{eq.reduite.semicl} by $w_n$, integrating by parts and taking the imaginary part, we obtain from (\ref{lim}) 

\EQ{\int a|w_n|^2dx=h_n \langle
F_n,w_n\rangle\rightarrow 0.\label{awn}} 
That implies that the measure $\nu=0$ in $supp(a)$. In addition, as supp($a$) satisfies the geometric control condition (GCC) then by the propagation property of $\nu$ (see  \cite{burq}),  we conclude that $\nu=0$ everywhere. Consequently, 
 \EQ{w_n\rightarrow 0  \text{ in } L^2_{loc}\label{cvloc}.}
Thanks to (\ref{cvloc}), (\ref{awn}) and taking into  account that $a\equiv 1$, for $|x|\geq r_0$, 
we deduce that $\int
|w_n|^2dx\rightarrow 0$ contradicting the fact that  $\|
w_n\|_{L^2}=1$.
\end{proof}
Combining Propositions \ref{moy.freq} and \ref{high.freq}, we have
the following result
\begin{prop}
For any $\beta>0$, there exist two positive constants $b$ and $C$ such that, for any $\lambda \in
D_{\beta,b}=\{\lambda\in\mathbb{C}; |Re\lambda|\leq b\text{ and
}|Im\lambda|\geq \beta\}$
we have
 \EQ{\|(\lambda
I-A)^{-1}\|\leq C.\label{hautf}}

\end{prop}

\begin{proof}
Set $\lambda=\mu+is $. In Propositions \ref{moy.freq} and \ref{high.freq}, we have shown that for all $\beta>0$ there exists $M>0$ such that, for every $|s|>\beta$, we have

\EQ{\|(is-A)^{-1}\|\leq M.} 
Moreover, the following trivial identity  holds 
\EQ{(is-A)=(\mu+is-A)-\mu.}
Multiplying both sides by
$(\mu+is-A)^{-1}(is-A)^{-1}$, we obtain
\EQ{(\mu+is-A)^{-1}=(is-A)^{-1}-\mu (\mu+is-A)^{-1}(is-A)^{-1},}
giving,\EQ{\|(\mu+is-A)^{-1}\|\leq M+\mu M\|(\mu+is-A)^{-1}\|.} 
Now the assumption $|\mu|\leq \frac{1}{2M}$ implies
\EQ{\|(\mu+is-A)^{-1}\|\leq M/2}
as desired. \end{proof}

\section{Local energy decay}
To prove the local energy decay (\ref{loc.energy}), we need the following Lemma that can be found for example in \cite{Shib,Shib-Dan}.
\begin{lem}\label{lemSh}
Let $E$ be a Banach space with norm $|.|_E$. Let $k\geq 0$ be an
integer and $\sigma$ be a positive number $\leq 1$. Assume that
$f\in \mathcal{C}^{k+\sigma}(\mathbb{R},E)$, and put
\EQ{F(t)=\frac{1}{2\pi}\int_{-\infty}^\infty f(\tau)exp(i\tau
t)d\tau.} Then \EQ{|F(t)|_E\leq
C(1+|t|)^{-(k+\sigma)}\| f\|_{k+\sigma,E}}

\end{lem}
\begin{proof}[Proof of Theorem 1.]
Let $\eta (t)\in C^\infty_0 (\mathbb{R})$ such that $\eta (t)=0$ for
$t\leq 1$ and $\eta (t)=1$ for $t\geq 2$. For any $f\in \mathcal H$, set $V(t)=\eta (t)
U(t)f$. Then we have, 
\EQ{(\partial _t-A)V=\eta '(t)U(t)f.\label{stat}
} 
First, the Fourier transform of $V$ satisfies \EQ{(\lambda-A)\widehat{V}=\widehat{\eta
'Uf}=x(\lambda)
\label{stat1}}
Second, from the finite speed of propagation, the operator 
\EQ{\widehat{\eta
'U}(\lambda):\mathcal{H}_{com}\rightarrow \mathcal{H}_{com} }
is analytic. Next,  (\ref{stat1}) yields
\EQ{\widehat{V}(\lambda)=(\lambda-A)^{-1}x(\lambda),\qquad \mathcal Re(\lambda)>0} and by the inverse formula of the Laplace transform, we have for $\alpha\in]0,\delta[$
\EQ{V(t)=\frac{1}{2i\pi}\int_{-\infty}^\infty
e^{(\alpha+is)t}(\alpha+is-A)^{-1}x(\alpha+is)ds,}
where $\delta$ given in lemma \ref{low.freq}. Next, let $\rho (s)\in
C^\infty_0(\mathbb{R})$, $\rho (s)=1$ if $|s|\leq \delta/2$ and equal
to 0 for $|s|>\delta$. We have for all $t>0$ \EQ{\begin{array}{lll}
V(t)&=&\displaystyle \frac{1}{2i\pi} e^{\alpha t}\left\{\int_{-\infty}^\infty
e^{ist}\rho (s)(\alpha+is-A)^{-1}x(\alpha+is)ds\right. \\&
&\displaystyle+\int_{-\infty}^\infty e^{ist}(1-\rho
(s))(\alpha+is-A)^{-1}x(\alpha+is)ds\Big\}\\&:=&J_1(t)+J_2(t).\end{array}}
Let us consider the term $J_1(t)$. Thanks to Lemma \ref{low.freq} and Lemma \ref{lemSh}, we deduce that for all $y\in \mathcal H$ the following holds
\begin{equation}
\begin{array}{lll}
|\langle \chi J_1(t),y\rangle|&\leq & e^{\alpha
t}C(1+t)^{-n/2}\| \rho
(.)(\chi R(\alpha+i.)x(\alpha+i.),y)\|_{N/2,{\R}}\\
&\leq & e^{\alpha t}CM_1(1+t)^{-n/2}\|f\|_{\mathcal H}\|y\|_{\mathcal H}.
\end{array}
\end{equation}
Now letting $\alpha\rightarrow 0$, we obtain
\EQ{\|\chi J_1(t)\|_{\mathcal H}\leq CM_1(1+t)^{-N/2}\|f\|_{\mathcal H}.\label{J1}}

Finally, we study the term $J_2(t)$. Using the identity
$e^{ist}=(it)^{-1}\frac{d}{ds}e^{ist}$,  an integrating by parts gives

\EQ{\begin{array}{lll} J_2(t)&=& \frac{1}{2i\pi} e^{\alpha t}
\Big\{\frac{1}{it}[e^{ist}(1-\rho
(s))(\alpha+is-A)^{-1}x(\alpha+is)\Big]_{s=-\infty}^{\infty}\\
&& +\displaystyle \frac{-1}{(it)^2}\Big[e^{ist}\frac{d}{ds}{(1-\rho
(s))(\alpha+is-A)^{-1}x(\alpha+is)}\Big]_{s=-\infty}^{\infty}\\
&&+\ldots
+\displaystyle\frac{(-1)^{l-1}}{(it)^l}\Big[e^{ist}\frac{d^{l-1}}{ds^{l-1}}{(1-\rho
(s))(\alpha+is-A)^{-1}x(\alpha+is)}\Big]_{s=-\infty}^{\infty}\\
&&\displaystyle+\frac{(-1)^{l}}{(it)^l}\int_{-\infty}^{\infty}e^{ist}\frac{d^{l}}{ds^{l}}\Big\{(1-\rho
(s))(\alpha+is-A)^{-1}x(\alpha+is)\Big\}ds\Big\}.\end{array}}
The boundedness  of the resolvent implies

\EQ{\|\frac{d^{j}}{ds^{j}}(\alpha+is-A)^{-1}\|\leq
j!M_a^j\|(\alpha+is-A)^{-1}\|.\label{est1}} 
Moreover, from (\ref{hautf}) and tha fact that $x(\lambda)$ is smooth and rapidly decaying function, we first deduce that
\EQ{\Big[e^{ist}\frac{d^{j}}{ds^{j}}{(1-\rho
(s))(\alpha+is-A)^{-1}x(\alpha+is)}\Big]_{s=-\infty}^{\infty}= 0
}
yielding 
\EQ{J_2(t)=\frac{(-1)^{l}e^{\alpha t}}{2i\pi
(it)^l}\int_{-\infty}^{\infty}e^{ist}\frac{d^{l}}{ds^{l}}\Big\{(1-\rho
(s))(\alpha+is-A)^{-1}x(\alpha+is)\Big\}ds,}
and therefore
\EQ{\|J_2(t)\|_{\mathcal H}\leq \frac{C e^{\alpha t}}{
t^l}\sum_{j,k}\int_{|s|\geq
d/2}\Big\|\frac{d^{j}}{ds^{j}}((\alpha+is-A)^{-1})\frac{d^{k}}{ds^{k}}x(\alpha+is)\Big\|_{\textcolor{red}{\mathcal H}}ds.
}
From (\ref{est1})  and (\ref{hautf})  we obtain \EQ{\|J_2(t)\|_{\mathcal H}\leq \frac{C e^{\alpha t}}{
t^l}\sum_{k}\int_{-\infty}^{\infty}\Big\|\frac{d^{k}}{ds^{k}}x(\alpha+is)\Big\|_{\mathcal H}ds.}
Since $x(\lambda)=\widehat{\eta ' u}$, then $\displaystyle\frac{d^{k}}{ds^{k}}x(\alpha+is)=\widehat{t^k\eta ' u}$, and by Plancherel theorem, we conclude that \EQ{\|J_2(t)\|_{\mathcal H}\leq \frac{C
e^{\alpha t}}{ t^l}\sum_{k}\int_{-\infty}^{\infty}e^{\alpha
z}\|z^k\eta '(z) U(z)f\|_{\mathcal H}dz\leq \frac{C e^{\alpha t}}{ t^l}\|f\|_{\mathcal H}.}
Letting $\alpha\rightarrow 0$ we get, for all $l\geq 1$
 \EQ{\|J_2(t)\|_{\mathcal H}\leq
 \frac{C }{ t^l}\|f\|_{\mathcal H}\label{J2}.}
Combining (\ref{J1}) and (\ref{J2}) we obtain \EQ{\|\chi V(t)\|_{\mathcal H}\leq
 C(1+t)^{-N/2}\|f\|_{\mathcal H}}
 as desired.
 \end{proof}
\section{Total energy decay}
From the previous section, we concluded that for any solution $u$ of \eqref{DLW} with initial data $(u_0,u_1)\in\mathcal{H}_{r_1}=H^D_{r_1}\times L^2_{r_1}$ we have \EQ{E_{r_1}(u,t)\leq c
(1+t)^{-N}.\label{loc.ener.decay2}} 
Now our aim is to prove that the total energy and the $L^2$-norm of
the solutions decay in time as in \eqref{tot.energy} and \eqref{tot.L^2}.

First, let us give some elementary known Lemmas that we need.
\begin{lem}\label{lem1}
Let $a,b>0$ such that $max (a,b)>1$. There exists a constant $C$
depending only on $a$ and $b$ such that the following inequality
holds: \EQ{\int_0^t(1+t-s)^{-a}(1+s)^{-b}\leq C (1+t)^{-min(a,b)}.}
In particular, for any $\alpha>0$ we have \EQ{\int_0^te^{-\alpha(t-s)}(1+s)^{-b}\leq C
(1+t)^{-b}.}
\end{lem}
For example, we refer to \cite{Segal} for the proof of Lemma \ref{lem1}.

\begin{lem}(Proposition 3.2 \cite{Kaw-Nak-Ono})\label{lem2}
Let $(u_0,u_1)\in (H^1\cap L^q)\times (L^2\cap L^q$), $1\leq q \leq
2$, and let $f\in L^\infty_{loc}([0,\infty);L^p\cap L^2)$, $1\leq
p \leq 2$. Then, the solution $u(t)$ of 
\EQ{
 u_{tt} + u_t -\De u  =  f(t,x)\label{free.Eq}}
with initial data $(u_0,u_1)$, satisfies
\begin{equation}
\begin{array}{lll}\|u(t)\|_{L^2} &\leq &
c(1+t)^{-N(1/q-1/2)/2}(\|u_0\|_q+\|u_1\|_q)+ce^{-\beta
t}(\|u_0\|_2+\|u_1\|_2)\\ &+ &\displaystyle
c\int_0^t(1+t-s)^{-N(1/p-1/2)/2}\|f(s)\|_pds+c\int_0^te^{-\beta(t-s)}\|f(s)\|_2ds
\end{array}
\end{equation}
with some $0<\beta<1/2$.
\end{lem}

\begin{lem}(Proposition 4.2 \cite{Kaw-Nak-Ono})\label{lem3}
Let $u$ be the solution of (\ref{free.Eq}). Assume that \EQ{\|u(t)\|_{L^2}^2\leq k_0(1+t)^{-a}} and
\EQ{\|f(t)\|_{\textcolor{red}{L^2}}^2\leq
k_1\{(1+t)^{-b}+(1+t)^{-c}E(t)^\mu+(1+t)^{-d}E(t)\}} with some
$k_0,k_1,a\geq 0$, $b,c,d>0$ and $0<\mu<1$. Then, $E(t)$ has the
decay property \EQ{E(t)\leq c_1(1+t)^{-\theta}} where $c_1$ is a
constant depending on $E(0)$ and other known constants, and
$\theta>0$ is given by
\EQ{\theta=min\{1+a,b,\frac{a+b}{2},\frac{c}{1-\mu},\frac{a+c}{2-\mu},a+d\}.}
\end{lem}
Let us now prove Theorem \ref{main thm 2}.
Let $u$ be a solution of (\ref{DLW}) with initial data $(u_0,u_1)\in \mathcal{H}_{r_1}=H^D_{r_1}\times L^2_{r_1}$ and consider a smooth function $\phi$ satisfying
$\phi (x)=1$ for $|x|\geq r_1$ and $\phi (x)=0$ for $x\in
\Omega_{r_0}.$ Setting $w=\phi (x)u$,  we have \EQ{
 w_{tt} + w_t -\De w  = -\nabla \phi .\nabla u-\De \phi u\equiv f(t,x).\label{Exterior.Eq}}
It is easy to see that $\|f(t)\|_{L^2}^2\leq c E_{r_1}(u,t)$. But thanks to Theorem \ref{main thm 1}, we have
\EQ{E_{r_1}(u,t)\leq c (1+t)^{-N}\label{loc.ener.decay1}.} So combining Lemma \ref{lem1} (with
$a=N/4$ and $b=N/2$) and Lemma \ref{lem2} (with
$p=q=1$) we obtain \EQ{\|w(t)\|_{L^2}^2\leq
c(1+t)^{-N/2}.\label{ext.L^2.decay}} Now, applying Lemma \ref{lem3} with  $a=N/2$ and $b=N$, we get \EQ{E(w,t)\leq
c_1(1+t)^{-\theta},\label{ext.ener.decay}} with
\EQ{\theta=min\{1+\frac{N}{2},\frac{3N}{4}\}.} So we can conclude
from  (\ref{loc.ener.decay1}), (\ref{ext.L^2.decay}) and
(\ref{ext.ener.decay}) that \EQ{\|u(t)\|_{L^2}^2\leq
c(1+t)^{-N/2}\label{glob.L^2.decay}} and \EQ{E(u,t)\leq
c_1(1+t)^{-\theta},\label{glob.ener.decay}}
as desired.
\begin{acknowledgement}
This paper has been supported by Deanship of Scientific Research of the Northern Border University under the reference no : 434-034.\\
The authors would like to thank the anonymous referee for his helpful comments on the presentation of the paper.
\end{acknowledgement}

\end{document}

%% file: mymacros.tex
\newcommand{\mb}[1]{\ensuremath{\mathbb{#1}}}

\newcommand{\R}{{\mb{R}}}

\DeclareMathSymbol{\intopmod}{\mathop}{symbols}{115}

